\documentclass[12pt,a4paper]{amsart}

\title{Absorbing time asymptotics in the oriented swap process}

\author{Alexey Bufetov \and Vadim Gorin \and Dan Romik}


\usepackage{euscript,amsfonts,amssymb,amsmath,amscd,amsthm,enumerate,hyperref,mathrsfs}

\usepackage{tikz,bm,float}
\usetikzlibrary{calc}
\colorlet{lgray}{white!85!black}
\colorlet{lred}{white!75!red}

\usepackage{comment}

\usepackage{graphicx}

\usepackage{color}

\usepackage[margin=1.1in]{geometry}
\newtheorem{theorem}{Theorem} 
\newtheorem*{theorem*}{Theorem}
\newtheorem{lemma}[theorem]{Lemma}

\newtheorem{proposition}[theorem]{Proposition}
\newtheorem{conjecture}[theorem]{Conjecture}
\newtheorem{corollary}[theorem]{Corollary}

\theoremstyle{remark}
\newtheorem{remark}[theorem]{Remark}

\numberwithin{equation}{section} \numberwithin{theorem}{section}

\newcommand{\Z}{\mathbb Z}

\newcommand{\1}{\mathbf 1}
\newcommand{\2}{\mathbf 2}
\newcommand{\N}{\mathbf N}

\renewcommand{\H}{\mathcal H}

\newcommand{\eps}{\varepsilon}

\newcommand{\ges}{{\scriptscriptstyle \geqslant}}

\begin{document}

\maketitle

\begin{abstract}
The oriented swap process is a natural directed random walk on the symmetric group that can be interpreted as a multi-species version of the Totally Asymmetric Simple Exclusion Process (TASEP) on a finite interval. An open problem from a 2008 paper of Angel, Holroyd, and Romik asks for the limiting distribution of the absorbing time of the process. We resolve this question by proving that this random variable satisfies GOE Tracy-Widom asymptotics. Our starting point is a distributional identity relating the behavior of the oriented swap
process to last passage percolation, conjectured in a recent paper of Bisi, Cunden,
Gibbons, and Romik. The main technical tool is a shift-invariance
principle for multi-species TASEPs, obtained by exploiting recent results of Borodin,
Gorin, and Wheeler for the stochastic colored six-vertex model.    
\end{abstract}

\section{Introduction}

A \emph{sorting network} is a shortest path between the identity permutation $\1\2\dots\N$ and the reverse permutation $\N\dots \2\1$ in the Cayley graph of the symmetric group $\mathfrak S_N$ associated with the swaps $\tau(i)$, $1\le i\le N-1$, of adjacent letters at positions $i$ and $i+1$. Equivalently, a sorting network can be encoded as a sequence of ${N\choose 2}$ indices
$$
 \left(s_1,s_2,\dots, s_{N\choose 2}\right),\quad s_i\in \{1,2\dots,N-1\},
$$
such that
$$
\tau(s_1) \cdot \tau(s_2)\cdots \tau(s_{N\choose 2})=\N\dots \2\1.
$$
We are interested in the asymptotic behavior of \emph{random} sorting networks as $N\to\infty$. There are at least two natural ways to introduce randomness here. One way is to consider a \emph{uniformly} random sorting network (out of the finite set of those for fixed~$N$). The rich asymptotic behavior of that model has been discussed in great detail in \cite{AHRV,AH,AGH,Roz,GR,ADHV,DVi,D}. Another natural way to introduce randomness was suggested in \cite{AHR} under the name ``oriented swap process''; it has a natural interpretation as an interacting particle system equivalent to a multi-species version of the Totally Asymmetric Simple Exclusion Process (TASEP).  We follow this last way and we now describe it.

In addition to the sequence of swap indices $\{s_i\}$, we consider a growing sequence of random numbers
$$
0<t_1<t_2<\dots< t_{N\choose 2}.
$$
We interpret $t_i$ as the time when the swap $\tau(s_i)$ happens: shortly before the time $t_i$ we observe the permutation $\tau(s_1) \cdot \tau(s_2)\cdots \tau(s_{i-1})$, and at time $t_i$ the next swap $\tau(s_i)$ is appended to the product. This results in a permutation-valued, continuous-time, process $(\sigma_t)_{t\ge 0}$, which can be interpreted as the evolution of a system of particles with labels (or colors) $1,\ldots,N$ interacting on the discrete interval $[N]=\{1,\ldots,N\}$, where $\sigma_t(k)$ is the label of the particle in position $k$ at time $t$. The initial condition $\sigma_0$ is the identity permutation $\1\2\dots\N$.

The random pair of sequences $\{s_i\}$ (swap positions) and $\{t_i\}$ (swap times) are generated inductively as follows:  let $\Pi_k$, $k=1,\ldots,N-1$ denote $N-1$ independent exponential clocks (rate $1$ Poisson point processes), and let $i=1$ and $t_0=0$. When the clock $\Pi_k$ is the first among the clocks for which $\sigma_{t_{i-1}}(k)<\sigma_{t_{i-1}}(k+1)$ to ring at some time $t>t_{i-1}$, we set $t_i$ equal to $t$, set $s_i = k$, and increase $i$ by 1.

The particle system interpretation of this definition is: whenever one of the Poisson clocks $\Pi_k$ rings, check the current labels of the particles at positions $k$ and $k+1$. If the one at $k+1$ has smaller label, then nothing happens. Otherwise, swap the labels of the particles at positions $k$ and $k+1$.
Clearly, after an almost surely finite time we will make all possible swaps and arrive at the reverse permutation $\N\dots \2\1$.

The authors of \cite{AHR} proved many results about the oriented swap process and its asymptotic behavior as the size $N$ of the system goes to infinity. Among the quantities they considered were certain random times at which different aspects of the process terminate. Specifically, define the $(N-1)$--dimensional vector $\mathbf{U}_N=(U_N(1),\dots,U_N(N-1))$, where for each $1\le k\le N_1$, $U_N(k)$ is the last time $t_i$ at which the swap $s_i=k$ happens. We refer to this random variable as the \emph{last swap time} associated with positions $k,k+1$; see Figure~\ref{Fig_sort}.

 \begin{figure}[t]
\begin{center}
{\scalebox{0.8}{\includegraphics{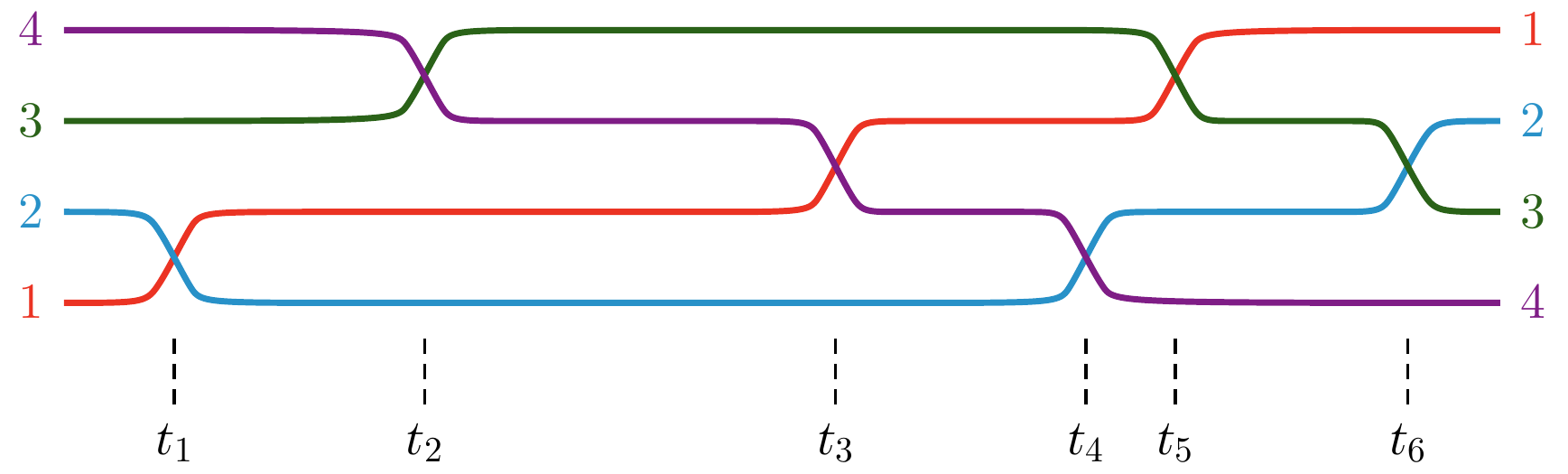}}}
 \caption{One possible sorting network or evolution of the particles in the oriented swap process. Here $N=4$, and $\mathbf{U}_4$ is the vector $(t_4,t_6,t_5)$.  \label{Fig_sort}}
\end{center}
\end{figure}

\begin{theorem}[\cite{AHR}] \label{Theorem_AHR}
Let a sequence $k=k(N)$, $N=1,2,\dots$, be given such that $\eps<k/N<1-\eps$ for some fixed $\eps>0$ and all sufficiently large $N$.
Denote $\gamma_y=1+2\sqrt{y(1-y)}$, and let $F_2$ denote the $\beta=2$ Tracy-Widom distribution.
Then we have the convergence in distribution
\begin{equation}
\label{eq_last_swap_limit}
\frac{U_N(k)-N \gamma_{k/N}}{ N^{1/3} (\gamma_{k/N})^{2/3} \left(\tfrac{k}{N}(1-\tfrac{k}{N})\right)^{-1/6}}\xrightarrow[N\to\infty]{d} F_2.
\end{equation}
\end{theorem}
We recall that the distribution $F_2$ is the universal scaling limit for the largest eigenvalues of random complex Hermitian matrices of growing sizes.

We proceed further by considering the $k$th \emph{particle finishing time} $Z_N(k)$, which is the last time at which the particle with label $k$ moved. It can be related to last swap times through
$$ Z_N(k) = \max\bigl(U_N(k-1),\, U_N(k)\bigr), \quad 1\le k\le N,  $$
with the convention that $U_N(0)=U_N(N) = 0$.
\cite[Theorem 1.6]{AHR} shows that \eqref{eq_last_swap_limit} implies exactly the same limiting behavior for $Z_N(k)$.

\medskip
Among the times $U_n(k)$ and $Z_n(k)$, perhaps the most important one is the \emph{absorbing time}:
$$
T_N^{\mathrm{OSP}}=\max_{1\le k\le N-1} U_N(k)
=\max_{1\le k\le N} Z_N(k)
=t_{N\choose 2},
$$
which is the time at which the very last swap in the oriented swap process occurs and we reach the reverse permutation $\N\dots\2\1$. Theorem \ref{Theorem_AHR} implies that $T_N^{\mathrm{OSP}} \approx 2 N$ as $N\to\infty$ (the maximum is attained for $k\approx N/2$, since $y\mapsto \gamma_y$ takes its maximum value of $2$ at $y=1/2$). However, the authors of \cite{AHR} could not identify the size of the fluctuations of $T_N^{\mathrm{OSP}}$ around $2N$ or their distributional limit; they stated this as an open problem in \cite[Section 8]{AHR}. The problem is also mentioned as a ``five coffee cup'' exercise in \cite[Exercise~5.22(e), p.~331]{romik-lis}.

The following theorem settles this problem, and is our main result.

\begin{theorem} Let $F_1$  be the $\beta=1$ Tracy-Widom distribution. We have
\label{Theorem_main}
\begin{equation}
\label{eq_abs_limit}
 \frac{T_N^{\mathrm{OSP}}-2N}{2^{1/3}\, N^{1/3} }\xrightarrow[N\to\infty]{d} F_1.
\end{equation}
\end{theorem}

We recall that the distribution $F_1$ is the universal scaling limit for the largest eigenvalues of random real symmetric matrices of growing sizes.

Our proof of Theorem \ref{Theorem_main} belongs to a recent circle of ideas (see \cite{BB, BGW, BCGR,D2,Ga}) on the hidden symmetries in models of integrable probability and their universal limits. In this text we demonstrate how these ideas can be efficiently used to answer asymptotic questions about complicated stochastic systems.

A starting point for the current work is a distributional identity conjectured recently in \cite{BCGR}, which relates the random vector $\mathbf{U}_N$ to a certain random statistic defined in terms of the last passage percolation model with exponential weights. Specifically one defines a random vector $\mathbf{V}_N=(V_N(1),\dots,V_N(N-1))$ of last passage times in an oriented percolation model (the definition is given in \eqref{eq-def-vn} below), which turns out to be related to $\mathbf{U}_N$.

\begin{conjecture}[\cite{BCGR}]
\label{conj-eq-dist}
We have the equality in distribution of random vectors
\begin{equation} \label{eq-dist-rand-vects}
\mathbf{U}_N \stackrel{d}{=} \mathbf{V}_N \qquad (N\ge 2).
\end{equation}
\end{conjecture}
Let us emphasize that although the interplay between particle systems and passage times in percolation is somewhat well-known --- in particular the equality in distribution of one-dimensional marginals $U_N(k) \stackrel{d}{=} V_N(k)$ for any $1\le k\le N-1$ follows easily from standard facts --- Conjecture \ref{conj-eq-dist} goes much further and does not seem to follow in a straightforward way from any known bijections. In addition to formulating the conjecture, the authors of \cite{BCGR} gave a computer-assisted verification of the distributional identity \eqref{eq-dist-rand-vects} for the initial values $2\le N\le 6$, which provides good evidence of its validity for general values of $N$.

Simultaneously, \cite{BCGR} observed that Conjecture \ref{conj-eq-dist} can be combined with known asymptotic results to yield \eqref{eq_abs_limit}.
We \emph{do not prove} Conjecture \ref{conj-eq-dist} in this text, as the generality of our present method seems to go in a slightly different direction.

Instead, we consider the maximal coordinate of the vector $\mathbf{V}_N$:
$$
T_N^{\mathrm{LPP}} = \max_{1\le k\le N-1} V_N(k).
$$
Below we represent $T_N^{\mathrm{LPP}}$ in terms of a multi-colored version of the TASEP, use the ideas of \cite{AHR}, which relate the TASEP dynamics on finite and infinite lattices, and add to them a recently discovered shift-invariance phenomenon \cite{BGW}
for the colored six-vertex model (which can be degenerated into multi-species/colored TASEP and thereby related to the oriented swap process). The ultimate result is the following equality in distribution of random variables for any $N\ge2$:
\begin{equation} \label{eq-dist-rand-vars}
T_N^{\mathrm{OSP}} \stackrel{d}{=} T_N^{\mathrm{LPP}}.
\end{equation}

The identity \eqref{eq-dist-rand-vars}, which can be thought of as a weakened version of Conjecture~\ref{conj-eq-dist},  allows us to use the known asymptotic results for $T_N^{\mathrm{LPP}}$ to deduce Theorem \ref{Theorem_main}.

\subsection*{Acknowledgements}
The authors thank the Institute for Pure and Applied Mathematics (IPAM) at UCLA for its hospitality during their visit there in February 2020, where some of the ideas contained in the current work were discussed. We also thank Leonid Petrov for helpful discussions.

A.B.\ was partially supported by the Deutsche Forschungsgemeinschaft (DFG, German Research Foundation) under
Germany’s Excellence Strategy - GZ 2047/1, Projekt ID 390685813. V.G.\ was partially supported by NSF grants DMS-1664619, DMS-1855458, by the NEC Corporation Fund for Research in Computers and Communications, and by the Office of the Vice Chancellor for Research and Graduate Education at the University of Wisconsin--Madison with funding from the Wisconsin Alumni Research Foundation.
D.R.\ was supported by the National Science Foundation grant No. DMS-1800725.

\section{Shift invariance}

\label{sec_shift_invariance}

An important technical ingredient of our proof is the shift-invariance for the colored (or multi-species) TASEP, which we now describe.

We deal with the colored TASEP on $\mathbb Z$. By definition, this is a time-dependent assignment $\zeta_t$ of integer labels (or colors) to points of $\mathbb Z$. At time $0$ we have $\zeta_t(i)=i$, $i\in\mathbb Z$. Further, each edge $(k,k+1)$ has an independent exponential clock (rate $1$ Poisson process) attached to it. Whenever the clock rings at time $t$, we check whether $\zeta_{t-}(k+1)<\zeta_{t-}(k)$. If so, then nothing happens. Otherwise, we swap the colors at $k$ and $k+1$ by setting $\zeta_{t+}(k+1):=\zeta_{t-}(k)$ and $\zeta_{t+}(k):=\zeta_{t-}(k+1)$. Then the clock is restarted and we proceed further. We refer to \cite[Section 3]{AHR} for the description of how the process $\zeta_t$, $t\ge 0$ can be constructed using the \emph{graphical representation.}

Let us remark that if we fix some $k$ and identify the colors $\le k$ calling them ``particles'' and identify the colors $>k$ calling them ``holes'' by setting $\nu_t^k(x)=1_{\zeta_t(x)\le k}$, then $\nu_t^k$ becomes the usual TASEP with particles jumping to the right at rate $1$. The initial configuration $\nu_0^k(x)=1_{x\le k}$ is then known as the \emph{step} initial condition. In this way, the colored TASEP becomes a coupling of a countable system of ordinary TASEPs, each one started from a (shifted) step initial condition.

We study $\zeta_t$ though its \emph{height functions}, which are a collection of random variables parameterized by $A,B\in\mathbb Z$. We define
\begin{equation}
\label{eq_A_heights}
h_{\le A \to \ge B} (t) :=  \# \left\{ x\in\Z \,:\, x \ge B \text{ and } \zeta_t (x) \le A \right\}, \qquad A,B \in \Z.
\end{equation}
In words, $h_{\le A\to \ge B}$ counts the number of colors $\le A$ at positions $\ge B$ at time $t$. Note that $h_{\le A\to \ge B}$ can take arbitrary large (but almost surely finite) values.

Let us also introduce another set of height functions for the convenience of matching the notation of \cite{BGW}.
$$
 \H^{\ges i}_{\text{TASEP}}(t,y):= \# \left\{ x\in\Z \,:\, x<y\text{ and } \zeta_t(x)\ge i\right\}, \quad t\ge 0,\, y\in \Z+\tfrac12.
$$

\begin{lemma} We have an almost sure identity
\begin{equation}
\label{eq_height_modification}
h_{\le A \to \ge B} (t)= \H^{\ges A+1}_{\textnormal{TASEP}}(t, B-1/2)+ (A-B+1).
\end{equation}
\end{lemma}
\begin{proof}
 At time $0$ with the notation $(u)_+=\max(u,0)$, we have
 $$
  h_{\le A \to \ge B}(0) = (A-B+1)_{+},\qquad \H_{\text{TASEP}}^{\ges A+1}(0, B-1/2)=(B-A-1)_+.
 $$
 Hence, \eqref{eq_height_modification} holds at $t=0$.
 Next, note that both $h_{\le A \to \ge B}(t)$ and $\H^{\ges A+1}_{\text{TASEP}}(t, B-1/2)+ (A-B+1)$ are monotone functions of $t\ge 0$. They both increase by $1$ whenever at time $t$ we have a swap, interchanging a color $\le A$ at $B-1$ with a color $\ge (A+1)$ at $B$. We conclude that \eqref{eq_height_modification} holds at all times.
\end{proof}

We can now state the shift-invariance result for the colored TASEP.

\begin{theorem} \label{Theorem_shift}
  Choose an index $\iota\ge 1$, color cutoff levels $k_1\dots,k_n\in \mathbb Z$, a time $t\ge 0$ and a collection of observation points  $y_1,\dots, y_n\in\mathbb Z+\tfrac12$. Set
    $$
    k_j'=\begin{cases} k_j,& j\ne \iota,\\ k_\iota+1, & j=\iota,\end{cases} \qquad \qquad  y'_j=\begin{cases}y_j, & j\ne \iota, \\  y_\iota + 1, & j=\iota. \end{cases}
    $$
    Assume that
    $$
    k_1\le k_2\le \dots\le k_n, \qquad  k'_1\le k'_2\le \dots\le k'_n,
    $$
   $$
     y_1\ge y_2\ge \dots\ge y_n, \qquad y'_1\ge y'_2\ge \dots\ge y'_n,
    $$
    Then the distribution of the vector of height functions
\begin{equation} \label{tasep-shiftinv-vect1}
    \bigl(\H_{\rm{TASEP}}^{\ges k_1}(t, y_1),\, \H_{\rm{TASEP}}^{\ges k_2}(t,y_2),\, \dots, \H_{\rm{TASEP}}^{\ges k_n}(t,y_n)\bigr)
\end{equation}
    coincides with the distribution of a similar vector with shifted $\iota$-th point and cutoff
\begin{equation} \label{tasep-shiftinv-vect2}
    \bigl(\H_{\rm{TASEP}}^{\ges k'_1}(t,y_1'),\, \H_{\rm{TASEP}}^{\ges k'_2}(t,y_2'),\, \dots, \H_{\rm{TASEP}}^{\ges k'_n}(t,y_n')\bigr).
\end{equation}
\end{theorem}
\begin{proof}
 By shifting the coordinate system, if necessary, we can assume without loss of generality that $k_1\ge 0$. Then a version of Theorem \ref{Theorem_shift} for the colored stochastic six-vertex model was proven in \cite[Theorem 1.2]{BGW} (see also \cite[Theorem 1.5]{Ga}). The latter model is an assignment of configurations (six types of vertices) to the points of the positive quadrant $\mathbb Z_{\ge 0}\times \mathbb Z_{\ge 0}$ by a sequential stochastic rule, which can be thought of as a multiparameter discrete time asymmetric simple exclusion process. There exists a limit transition from the stochastic vertex model to TASEP, which was noticed (for the colorless model) in
 \cite[Section 2.2]{BCG} and proved in detail and greater generality in \cite{Aggarwal}. For the colored version the proof is the same.

 In the notation of \cite[Section 1.2 and Figure 4]{BGW}, we assume that the hopping probabilities $b_1$ and $b_2$ are homogeneous (do not depend on the lattice point) and set
 $$
  b_1=\eps, \quad b_2=0,
 $$
 Then the height function of the stochastic six-vertex model, which we denote by $\H^{\ges k}_{\text{6v}}(\cdot, \cdot)$, converges in distribution to that of a TASEP height function:
  \begin{equation}
  \label{eq_6v_to_TASEP}
    \H^{\ges k}_{\text{6v}}(\lfloor \eps^{-1} t\rfloor -y, \lfloor \eps^{-1} t\rfloor +y) \xrightarrow[\eps\to 0+]{d} \H^{\ges k}_{\text{TASEP}}(t,y).
  \end{equation}
  This is to be understood in the sense that the convergence in \eqref{eq_6v_to_TASEP} holds for the distributions of finite (arbitrary) collections of values of $(k,t,y)$. Given this relation, the statement of Theorem \ref{Theorem_shift} is a direct consequence of  \cite[Theorem 1.2]{BGW}.
\end{proof}
\begin{remark}
 At first, it might seem that  \cite[Theorem 1.2]{BGW} ought to imply a more general statement than the one we formulated: indeed, that theorem allowed shifts in the situation when the observation points $(x_j,y_j)$ are not restricted to a single line --- in the case of the TASEP, an analogous statement would mean accessing the heights at different values of the time parameter~$t$. However,  \cite[Theorem 1.2]{BGW} required certain ordering inequalities for the points $(x_j,y_j)$, and the only way for these inequalities to be satisfied in the limit \eqref{eq_6v_to_TASEP} is by making all times equal, as in
 \eqref{tasep-shiftinv-vect1}--\eqref{tasep-shiftinv-vect2}.
That is one reason why at this point we are unable to give a full proof of Conjecture~\ref{conj-eq-dist} --- we will only prove in the next section its particular case corresponding to all equal times in TASEP. At the same time, since the results of \cite{BCGR} support the full validity of the conjecture in its stronger form, one can wonder if Theorem~\ref{Theorem_shift} might also have as yet unknown extensions involving unequal times.
\end{remark}

We end this section by restating a particular case of Theorem \ref{Theorem_shift} in terms of the height functions $h_{\le A \to \ge B} (t)$ of \eqref{eq_A_heights}.

\begin{corollary}
\label{cor:shift}
Fix $N \in \Z_{\ge 1}$.  We have a distributional identity of $(N-1)$--dimensional vectors
\begin{multline}
\label{eq_TASEP_shift_cor}
\left( h_{\le 1 \to \ge N} (t), h_{\le 2 \to \ge N-1} (t), \dots, h_{\le N-1 \to \ge 2} (t)  \right) \\ \,{\buildrel d \over =}\, \left( h_{\le N-1 \to \ge 2N-2} (t), h_{\le N-1 \to \ge 2N-4} (t), \dots, h_{\le N-1 \to \ge 2} (t)  \right)
\end{multline}
\end{corollary}
\begin{proof}
 Given the identity \eqref{eq_height_modification}, this follows by repeated applications of Theorem \ref{Theorem_shift} with $n=N-1$. Indeed,  the $(N-1)$st coordinates of the vectors in \eqref{eq_TASEP_shift_cor} are the same. The $(N-2)$nd coordinate is $h_{\le N-2 \to \ge 3} (t)$ for the left-hand side and $h_{\le N-1 \to \ge 4}(t)$ for the right-hand side. Hence, we can shift one into another by Theorem \ref{Theorem_shift} with $\iota=n-1=N-2$. Next, we apply Theorem \ref{Theorem_shift} twice for $\iota=n-2=N-3$, shifting $h_{\le N-3 \to \ge 4} (t)$ into $h_{\le N-1 \to \ge 6}(t)$.
 Continuing in this way for smaller values of $\iota$, we reach \eqref{eq_TASEP_shift_cor}.
\end{proof}

\section{The oriented swap process}

In this section we prove Theorem \ref{Theorem_main}.

\subsection{Coupling of TASEPs on different spaces}
\label{Section_coupling}
We need to gather some facts from \cite{AHR} about the connection between the oriented swap process and the colored TASEP on $\mathbb Z$.

We start by defining the colored TASEP on the finite set $[N]=\{1,2,\dots,N\}$. It is defined in exactly the same way as the colored TASEP on $\mathbb Z$, but all the particles stay in $[N]$: the swaps $(0,1)$ and $(N,N+1)$ are  prohibited. Clearly, this is just a particular representation of the oriented swap process from the introduction. In particular, the system stops at the random absorbing time $t_{N\choose 2} = T_N^{\mathrm{OSP}}$.

Following \cite[Section 3]{AHR} we introduce a coupling of the colored TASEP on $\mathbb Z$ (which we continue to denote by $\zeta_t$ as in Section~\ref{sec_shift_invariance}) with its counterpart on $[N]$, which we will denote by $\zeta_t^N(x)$. The coupling proceeds as follows: in order to construct the process on $\mathbb Z$ we need clocks (Poisson processes) attached to each edge $(k,k+1)$ --- whenever the clock rings, particles at $k$ and $k+1$ attempt to swap (and succeed only if the particle at $k+1$ had a larger label). For the process on $[N]$ we are going to use exactly the same $N-1$ clocks as the $N-1$ clocks of the process on $\mathbb Z$ corresponding to the edges $(1,2)$, $(2,3)$,\dots, $(N-1,N)$.

Now consider, for fixed $k\in\Z$,
$$
 \nu^k_t(x)= \begin{cases} 1, & \zeta_t(x)\le k,\\ 0,& \zeta_t(x)>k.\end{cases}
$$
Then $\nu^k_t$ is a realization of the usual TASEP on $\mathbb Z$ (with particles given by $1$'s and jumping to the right) started from a step initial condition: at time $0$ the particles are at $(\dots,k-2,k-1,k)$.

Similarly, we can define
$$
 \nu^{k,N}_t(x)= \begin{cases} 1, & \zeta_t^N(x)\le k,\\ 0,& \zeta_t^N(x)>k,\end{cases}
$$
and observe that $\nu^{k,N}_t$ is a realization of a TASEP on $[N]$, with no particles entering from the left and particles prohibited from exiting on the right, i.e., swaps along both edges $(0,1)$ and $(N,N+1)$ are blocked. Assuming $1\le k\le N$, in $\nu^{k,N}_0$ the particles occupy positions $1,2,\dots,k$.

Note that all of the processes $\nu_t^k$ for different values of $k$ almost surely take values in the subset $\Omega$ of the space of colorless TASEP configurations $\{0,1\}^\Z$ consisting of configurations with only a finite number of $1$'s to the right of the origin. We make use of the following result.

\begin{proposition}[{\cite[Lemma 3.3]{AHR}}] Define two combinatorial operators acting on $\Omega$: the ``cut-off'' operator $R_k$ keeps only the rightmost $k$ particles in a (potentially infinite) system of particles. The ``push-back'' operator $B_n$ pushes all the particles into the ray $(-\infty,n]$, preserving their order (and moving all particles by the minimal possible distances to the left). Then we have an almost sure identity
\begin{equation}
\label{eq_pushback}
 \nu^{k,N}_t= B_N R_k \nu^k_t,
\end{equation}
holding simultaneously for all $1\le k\le N$, and all $t$.
\end{proposition}

\subsection{The colored TASEP on $\Z_N$ and its height functions}

Define
$$
\hat{h}_{\le A \to \ge B} (t) :=  \# \left\{ x\in\Z \,:\, x \ge B \text{ and } \zeta_t^N(x)\le A  \right\}, \qquad 1\le A,B\le N,
$$
and note that these are the height functions associated with the colored TASEP on~$\Z_N$.

\begin{proposition} Under the coupling of colored TASEPs on different subsets of $\mathbb Z$ of Section \ref{Section_coupling}, we have an almost sure identity holding for all $t_1,\dots,t_{N-1}\ge 0$:
\begin{multline} \label{eq_1st_infinite_identity}
\left( \hat h_{\le 1 \to \ge N} (t_1), \hat h_{\le 2 \to \ge N-1} (t_2), \dots, \hat h_{\le N-1 \to \ge 2} (t_{N-1})  \right)
\\ \,=\, \left( \min \left( h_{\le 1 \to \ge N} (t_1),1 \right), \min \left( h_{\le 2 \to \ge N-1} (t_2), 2 \right), \dots, \min \left( h_{\le N-1 \to \ge 2} (t_{N-1}), N-1 \right) \right).
\end{multline}
\end{proposition}
\begin{proof}
 For an ordinary colorless TASEP $\nu_t^k$ (or $\nu_t^{k,N-1}$), let $H(x; \nu_t^k)$ (resp.\ $H(x; \nu_t^{k,N-1})$) denote the number of the particles \emph{strictly} to the right of $x$; this is a deterministic function of $\nu_t^k$ (resp.\ $\nu_t^{k,N-1}$). Then we have
 \begin{equation}
 \label{eq_x1}
   \hat h_{\le i \to \ge N+1-i} (t)= H(N-i; \nu_t^{i,N}), \quad i=1,2,\dots,N-1.
 \end{equation}
 Simultaneously, (cf.\ \cite[Equations (4) and (5)]{AHR})
 \begin{equation}
 \label{eq_x2}
  \min \left( h_{\le i \to \ge N+1-i} (t),i \right)= H(N-i, R_i \nu^i_t)= H(N-i, B_N R_i \nu^i_t), \quad i=1,2,\dots,N-1.
 \end{equation}
 For the first equality in the last formula, notice that since we deal with $\min(\cdot,i)$, we can ignore all the particles beyond the first $i$; for the second equality, notice that for $i$--particle configurations, $B_N$ does not change the number of particles to the right of $N-i$, since we have $i$ free spots to the right from $N-i$; these are $N+1-i$, $N+2-i$,\dots $N$.

 Applying \eqref{eq_x1}, \eqref{eq_x2}, and \eqref{eq_pushback} to each coordinate of the vector \eqref{eq_1st_infinite_identity}, we get the desired identity.
\end{proof}

Let $T_N^{\mathrm{OSP}}$ be the absorbing time, that is, the time when the colored TASEP on $[N]$ stops. The definition, identity \eqref{eq_1st_infinite_identity}, and Corollary \ref{cor:shift} imply that
\begin{align}
\label{eq:tauExit}
\mathrm{Prob} \big( &T_N^{\mathrm{OSP}} \le t \big) = \mathrm{Prob} \left( \hat h_{\le 1 \to \ge N} (t) =1, \hat h_{\le 2 \to \ge N-1} (t) =2, \dots, \hat h_{\le N-1 \to \ge 2} (t) =  N-1  \right)
 \\ &= \mathrm{Prob} \left( h_{\le 1 \to \ge N} (t) \ge 1, h_{\le 2 \to \ge N-1} (t) \ge 2, \dots, h_{\le N-1 \to \ge 2} (t) \ge N-1 \right)
\nonumber \\ &=\mathrm{Prob} \left( h_{\le N-1 \to \ge 2N-2} (t) \ge 1, h_{\le N-1 \to \ge 2N-4} (t) \ge 2, \dots, h_{\le N-1 \to \ge 2} (t) \ge N-1 \right).
\nonumber
\end{align}

\subsection{Proof of Theorem \ref{Theorem_main}}
\label{Section_proof}
Using the coupling from Section \ref{Section_coupling}, the event
$$ A_t^{N-1} := \left( h_{\le N-1 \to \ge 2N-2} (t) \ge 1, h_{\le N-1 \to \ge 2N-4} (t) \ge 2, \dots, h_{\le N-1 \to \ge 2} (t) \ge N-1 \right)$$
appearing on the right-hand side of \eqref{eq:tauExit}
has the following interpretation:
given a colorless TASEP started from the step initial condition with particles occupying the positions $N-1$,$N-2$,$N-3,\ldots$ at time $0$,
  $A_t^{N-1}$ is the event
  that at time $t$ the first particle is at or to the right of position $2N-2$, the second particle is at or to the right of position $2N-4$,\dots, the $(N-1)$th particle is at or to the right of position $2$.

We can now reinterpret this event in terms of Last Passage Percolation (LPP) with exponential weights, using the well-known correspondence between the TASEP and LPP with such weights. We summarize this relationship between the two processes; for a more detailed explanation, see, e.g., the discussion around Figure 4 in \cite{BG_lectures}, or \cite[Section 4.7]{romik-lis}. In short, we treat the configuration of the TASEP as a broken line interface with particles representing segments of slope $-1$ and holes (i.e., the absence of a particle at some location) representing segments of slope $1$, as in Figure \ref{Fig_LPP}. Then the time evolution of the TASEP becomes the growth of the line interface, and the growth follows the rule that each inner corner is filled with a unit square after an exponential waiting time (independent of all other waiting times). These waiting times, in turn, form the array of weights for the last passage percolation model.

 \begin{figure}[t]
\begin{center}
{\scalebox{1.2}{\includegraphics{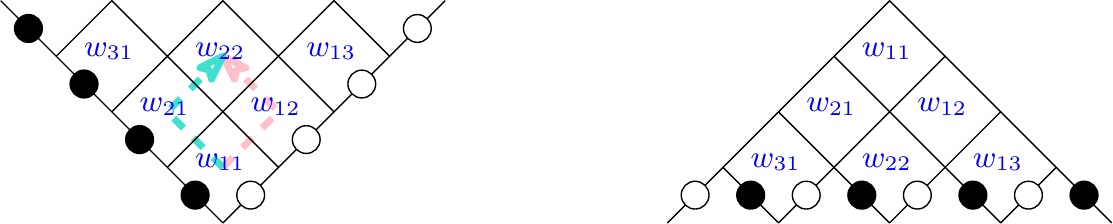}}}
 \caption{Left panel: a step initial condition in the TASEP corresponds to the wedge-type broken line interface. The shape of the interface after time $t$ can be computed using last passage times: the time when a box $(i,j)$ is added to the interface is $L(i,j)$. In particular, $L(2,2)=\max(w_{11}+w_{12}+w_{22}, w_{11}+w_{21}+w_{22})$. Right panel: by flipping the picture we can get the TASEP with a flat initial condition.  \label{Fig_LPP}}
\end{center}
\end{figure}

Making the idea more precise and applying it to our particular situation, we take a quadrant filled with i.i.d.\ exponential mean $1$ random variables $w_{ij}$, $i,j=1,2,\dots$ and draw it in Russian notation, as on the left panel of Fig.~\ref{Fig_LPP}. For $i,j\ge1$, we define the last passage time $L(i,j)$ associated with the square with coordinates $(i,j)$ by
$$
 L(i,j)=\max_{(1,1)=b[1]\to b[2]\to \dots \to b[m]=(i,j)}  \sum_{k=1}^m w_{b[k]},
$$
where $m=i+j-1$ and the maximum is taken over all monotone lattice paths joining $(1,1)$ with $(i,j)$ (i.e., paths with $b[k+1]-b[k]\in\{(0,1),(1,0)\}$ for all~$k$). The last passage time $L(i,j)$ represents the time when the unit square with coordinates $(i,j)$ was filled; in the TASEP picture (with our particular step initial condition offset by $N-1$ units from the usual one), this corresponds to the time it took the particle that started out in position $N-i$ to arrive at position $N-i+j$.

Now, the vector $\mathbf{V}_N$ appearing in \eqref{eq-dist-rand-vects} was defined in \cite{BCGR} in terms of the last passage percolation times as
\begin{equation}
\label{eq-def-vn}
\mathbf{V}_N = (L(1,N-1), \, L(2,N-2),\, \dots,\, L(N-1,1)).
\end{equation}
Moreover, with the correspondence described above, we now see that the event
$A_t^{N-1}$
is the same as the event
$$
 \bigl( L(1,N-1)\le t, \, L(2,N-2)\le t,\, \dots,\, L(N-1,1)\le t \bigr).
$$
Hence, we conclude that  \eqref{eq:tauExit} implies the equality in law
\begin{equation}
\label{eq_abs_as_LPP}
 T_N^{\mathrm{OSP}}  \stackrel{d}{=}\max\bigl( L(1,N-1), \, L(2,N-2),\, \dots,\, L(N-1,1)\bigr)
= \max_{1\le k\le N-1} V_N(k) =  T_N^{\mathrm{OSP}},
\end{equation}
proving \eqref{eq-dist-rand-vars}.

Finally, the asymptotics of $T_N^{\mathrm{LPP}}$ were established in \cite[Theorem 1.1]{BZ} (one needs to replace $2N+1$ by $N$ and take $\gamma=1/2$ there), which, in view of \eqref{eq-dist-rand-vars}, gives precisely \eqref{eq_abs_limit}.
\qed

\begin{remark}
An alternative way to derive the asymptotics of \eqref{eq_abs_as_LPP} is by flipping the picture vertically and computing instead the maximum
$$
\max\bigl( L( (2,N)\to (N,N)), \, L( (3,N-1)\to (N,N)),\, \dots,\, L((N,2)\to (N,N))\bigr),
$$
where $L((i,j)\to(i',j'))$ now denotes a more general last passage percolation time
$$
 L((i,j)\to(i',j'))=\max_{(i,j)=b[1]\to b[2]\to \dots \to b[m]=(i',j')}  \sum_{k=1}^m w_{b[k]},
 $$
with $1\le i\le i'$, $1\le j\le j'$, and the maximum being taken over all monotone lattice paths joining $(i,j)$ to $(i',j')$.
Using the correspondence between TASEP and last passage percolation again, one can then identify the latter maximum with the first time the height at $0$ for the TASEP started from a so-called \emph{flat} initial condition reaches the value $N$; see the right panel of Figure~\ref{Fig_LPP}. From this perspective, the asymptotic computation leading to the Tracy-Widom distribution $F_1$ goes back to \cite{S}, \cite{BFPS}. In a wider context, the first appearance of the $F_1$ distribution in a closely related framework dates to \cite{BR}.
\end{remark}

\bigskip
\noindent
Alexey Bufetov (\texttt{alexey.bufetov@gmail.com})
\\
Hausdorff Center for Mathematics \& Institute for Applied Mathematics, University of Bonn

\bigskip
\noindent
Vadim Gorin (\texttt{vadicgor@gmail.com})
\\
University of Wisconsin - Madison, USA; Massachusetts Institute of Technology, USA; Institute for Information Transmission Problems, Russia

\bigskip
\noindent
Dan Romik (corresponding author, \texttt{romik@math.ucdavis.edu})
\\
University of California, Davis

\end{document}